\documentclass{article}%
\usepackage{amsmath}
\usepackage{amssymb}%
\usepackage{amsfonts}%
\usepackage{graphicx}
\providecommand{\U}[1]{\protect\rule{.1in}{.1in}}
\newtheorem{theorem}{Theorem}

\newtheorem{corollary}[theorem]{Corollary}

\newtheorem{lemma}[theorem]{Lemma}

\newenvironment{proof}[1][Proof]{\noindent\textbf{#1.} }{\ \rule{0.5em}{0.5em}}
\begin{document}

\title{Tropical tensor product and beyond}
\author{Peter Butkovi\v{c} and Miroslav Fiedler}
\maketitle

\begin{abstract}
Our aim is to introduce the tropical tensor product and investigate its
properties. In particular we show its use for solving tropical matrix equations.

\end{abstract}

\section{Tropical prerequisites}

The aim of this paper is to introduce the tropical tensor product of matrices
and investigate its properties. In particular we show its use for solving
tropical matrix equations.

$\bigskip$The operations are in max-plus, that is for $a,b\in\overline
{\mathbb{R}}:=\mathbb{R}\cup\{\varepsilon=-\infty\}$ we define
\begin{align*}
a\oplus b  & =\max\left(  a,b\right)  ,\\
a\otimes b  & =a+b.
\end{align*}
Note that $\left(  \overline{\mathbb{R}},\oplus,\otimes\right)  $ is a
commutative idempotent semiring.

For matrices $A,B$ of compatible sizes we define%

\begin{align*}
A\oplus B  & =(a_{ij}\oplus b_{ij}),\\
A\otimes B  & =\left(  \sum_{k}^{\oplus}a_{ik}\otimes b_{kj}\right)  ,\\
\alpha\otimes A  & =\left(  \alpha\otimes a_{ij}\right)  .
\end{align*}
We denote $A\otimes A$ by $A^{2},$ etc... (including scalars).

An $n\times n$ matrix is called \textit{diagonal}, notation $diag(d_{1}%
,...,d_{n}),$ or just $diag(d),$ if its diagonal entries are $d_{1}%
,...,d_{n}\in\mathbb{R}$ and off-diagonal entries are $\varepsilon.$ The
matrix $diag(0,...,0)$ of an appropriate order will be called the
\textit{unit} \textit{matrix}\ and denoted by $I.$ Any matrix which can be
obtained from the unit (diagonal) matrix by permuting the rows and/or columns
will be called a \textit{permutation matrix }(\textit{generalized permutation
matrix}). Note that a matrix in the max-plus setting is invertible if and only
if it is a generalized permutation matrix \cite{PB book}.

We denote $N=\left\{  1,2,...,n\right\}  $ and by $P_{n}$ the set of all
permutation of $N.$ \textit{Tropical permanent} is an analogue of the
conventional permanent:%
\[
maper(A)\overset{df}{=}{\sum_{\pi\in P_{n}}}^{\oplus}{\prod_{i\in N}}%
^{\otimes}a_{i,\pi(i)}=\max_{\pi\in P_{n}}\sum_{i\in N}a_{i,\pi(i)}.
\]
Hence finding $maper(A)$ amounts to\ solving the classical assignment problem,
that is to finding a permutation $\pi\in P_{n}$ maximising $w(\pi,A)$ where%
\[
w(\pi,A)=\sum_{i\in N}a_{i,\pi(i)}.
\]

As a direct consequence of the Hungarian method for solving the assignment
problem \cite{Burkard AP} we have the following \cite{PB book}.

\begin{theorem}
\label{Th AP}Let $A\in\overline{\mathbb{R}}^{n\times n}$ and suppose that
$w(\pi,A)$ is finite for at least one $\pi\in P_{n}.$ Then diagonal matrices
$C,D$ such that
\begin{equation}
maper\left(  C\otimes A\otimes D\right)  =0\label{maper CAD}%
\end{equation}
and
\begin{equation}
C\otimes A\otimes D\leq0\label{CAD <=0}%
\end{equation}
exist and can be found in $O(n^{3})$ time. Also, the following holds for any
diagonal matrices $C$ and $D$ satisfying (\ref{maper CAD}):
\begin{align*}
maper(A)  & =\\
& =\left(  maper(C)\otimes maper(D)\right)  ^{-1}\\
& =\left(  {\prod_{i\in N}}^{\otimes}c_{ii}\otimes{\prod_{i\in N}}^{\otimes
}d_{ii}\right)  ^{-1}.
\end{align*}

\end{theorem}

\textit{The maximum cycle mean of }$A\in\overline{\mathbb{R}}^{n\times n}$ is%
\[
\lambda(A)=\max\left\{  \frac{a_{i_{1}i_{2}}+a_{i_{2}i_{3}}+...+a_{i_{k}i_{1}%
}}{k};k,i_{1},...,i_{k}\in N\right\}
\]
where

\begin{theorem}
\label{Th Max eproblem}For every $A\in\overline{\mathbb{R}}^{n\times n}$ there
is an $x\in\overline{\mathbb{R}}^{n},x\neq\varepsilon$ (eigenvector) such that
$A\otimes x=\lambda\left(  A\right)  \otimes x.$ If $x\in\mathbb{R}^{n}$ and
$A\otimes x=\lambda\otimes x$ then $\lambda=\lambda\left(  A\right)  .$ If $A$
is irreducible then all eigenvectors are finite and hence $\lambda\left(
A\right)  $ is the unique eigenvalue .
\end{theorem}

We also denote
\[
a\oplus^{\prime}b=\min(a,b),
\]%
\[
a\otimes^{\prime}b=a+b\text{ \ if }\{a,b\}\neq\{-\infty,+\infty\}
\]
and%
\[
\left(  -\infty\right)  \otimes^{\prime}\left(  +\infty\right)  =+\infty
=\left(  +\infty\right)  \otimes^{\prime}\left(  -\infty\right)  .
\]
The \textit{conjugate} of $A$ is $A^{\#}=-A^{T}.$ It is known that $A\otimes
x\leq b\Longleftrightarrow x\leq A^{\#}\otimes^{\prime}b\overset{df}{=}%
\overline{x}$ \cite{PB book} and consequently, a solution to $A\otimes x=b$
exists if and only if $\overline{x}$ is a solution.

\section{\bigskip Tensor product}

Let $A=\left(  a_{ij}\right)  \in\mathbb{R}^{m\times n},B=\left(
b_{ij}\right)  \in\mathbb{R}^{r\times s}.$ The\textit{\ tensor product} of $A$
and $B$ is the following $mr\times ns$ matrix:%
\[
A\boxtimes B=\left(
\begin{array}
[c]{cccc}%
A\otimes b_{11} & A\otimes b_{12} & \cdots & A\otimes b_{1s}\\
A\otimes b_{21} & A\otimes b_{22} & \cdots & A\otimes b_{2s}\\
\cdots & \cdots & \cdots & \cdots\\
A\otimes b_{r1} & A\otimes b_{r2} & \cdots & A\otimes b_{rs}%
\end{array}
\right)  .
\]

Note that $\left(  A\boxtimes B\right)  ^{T}=A^{T}\boxtimes B^{T}$ and the
tensor product of two diagonal matrices of order $n$ is a diagonal matrix of
order $n^{2}$. In particular the tensor product of two unit matrices is a unit matrix.

Matrices $A$ and $B$ are called \textit{product} \textit{compatible} if the
product $A\otimes B$ is well defined, that is the number of columns of $A$ is
equal to the number of rows of $B.$ The proof of the following theorem follows
the lines of the proof of the analogous statement in \cite{MF book}.

\begin{theorem}
\label{Th ABCD}If the matrices $A$ and $C$ are \textit{product }compatible and
also $B$ and $D$ are \textit{product }compatible then $A\boxtimes B$ and
$C\boxtimes D$ are \textit{product }compatible and%
\begin{equation}
\left(  A\boxtimes B\right)  \otimes\left(  C\boxtimes D\right)  =\left(
A\otimes C\right)  \boxtimes\left(  B\otimes D\right)  .\label{ABCD is ACBD}%
\end{equation}

\end{theorem}

\begin{proof}
If $B$ is $p\times q$ and $D$ is $q\times r$ then the LHS of
(\ref{ABCD is ACBD}) is%
\[
\left(
\begin{array}
[c]{cccc}%
A\otimes b_{11} & A\otimes b_{12} & \cdots & A\otimes b_{1q}\\
A\otimes b_{21} & A\otimes b_{22} & \cdots & A\otimes b_{2q}\\
\cdots & \cdots & \cdots & \cdots\\
A\otimes b_{p1} & A\otimes b_{p2} & \cdots & A\otimes b_{pq}%
\end{array}
\right)  \otimes\left(
\begin{array}
[c]{cccc}%
C\otimes d_{11} & C\otimes d_{12} & \cdots & C\otimes d_{1r}\\
C\otimes d_{21} & C\otimes d_{22} & \cdots & C\otimes d_{2r}\\
\cdots & \cdots & \cdots & \cdots\\
C\otimes d_{q1} & C\otimes d_{q2} & \cdots & C\otimes d_{qr}%
\end{array}
\right)  .
\]
Using blockwise tropical matrix multiplication we get that this is the same as%
\[
\left(
\begin{array}
[c]{ccc}%
A\otimes C\otimes\left(  b_{11}\otimes d_{11}\oplus...\oplus b_{1q}\otimes
d_{q1}\right)  & \cdots & A\otimes C\otimes\left(  b_{11}\otimes d_{1r}%
\oplus...\oplus b_{1q}\otimes d_{qr}\right) \\
\cdots & \cdots & \cdots\\
A\otimes C\otimes\left(  b_{p1}\otimes d_{11}\oplus...\oplus b_{pq}\otimes
d_{q1}\right)  & \cdots & A\otimes C\otimes\left(  b_{p1}\otimes d_{1r}%
\oplus...\oplus b_{pq}\otimes d_{qr}\right)
\end{array}
\right)  ,
\]
which is $\left(  A\otimes C\right)  \boxtimes\left(  B\otimes D\right)  .$
\end{proof}

\begin{theorem}
\bigskip If $A$ and $B$ are invertible then $A\boxtimes B$ is invertible too
and
\[
\left(  A\boxtimes B\right)  ^{-1}=A^{-1}\boxtimes B^{-1}.
\]

\end{theorem}

\begin{proof}
By Theorem \ref{Th ABCD} we have%
\[
\left(  A\boxtimes B\right)  ^{-1}\otimes\left(  A^{-1}\boxtimes
B^{-1}\right)  =\left(  A\otimes A^{-1}\right)  \boxtimes\left(  B\otimes
B^{-1}\right)  =I\boxtimes I=I.
\]

\end{proof}

\begin{theorem}
\label{Th Eproblem of TTP}If $A\otimes x=\lambda\otimes x$ and $B\otimes
y=\mu\otimes y$ then

\begin{enumerate}
\item[(a)] $\left(  A\boxtimes B\right)  \otimes\left(  x\boxtimes y\right)
=\lambda\otimes\mu\otimes\left(  x\boxtimes y\right)  $ and

\item[(b)] $\alpha\otimes\lambda\oplus\beta\otimes\mu$ is an eigenvalue of
$\alpha\otimes\left(  A\boxtimes I\right)  \oplus\beta\otimes\left(
I\boxtimes B\right)  $ for any $\alpha,\beta\in\mathbb{R}.$
\end{enumerate}
\end{theorem}

\begin{proof}
(a) By Theorem \ref{Th ABCD} we have%
\begin{align*}
\left(  A\boxtimes B\right)  \otimes\left(  x\boxtimes y\right)   & =\\
& =\left(  A\otimes x\right)  \boxtimes\left(  B\otimes y\right) \\
& =\left(  \lambda\otimes x\right)  \boxtimes\left(  \mu\otimes y\right) \\
& =\lambda\otimes\mu\otimes\left(  x\boxtimes y\right)  .
\end{align*}

(b) By Theorem \ref{Th ABCD} we have%
\begin{align*}
\left(  \alpha\otimes\left(  A\boxtimes I\right)  \oplus\beta\otimes\left(
I\boxtimes B\right)  \right)  \otimes\left(  x\boxtimes y\right)   & =\\
& =\alpha\otimes\left(  A\boxtimes I\right)  \otimes\left(  x\boxtimes
y\right)  \oplus\beta\otimes\left(  I\boxtimes B\right)  \otimes\left(
x\boxtimes y\right) \\
& =\alpha\otimes\left(  A\otimes x\right)  \boxtimes\left(  I\otimes y\right)
\oplus\beta\otimes\left(  I\otimes x\right)  \boxtimes\left(  B\otimes
y\right) \\
& =\alpha\otimes\left(  \lambda\otimes x\right)  \boxtimes y\oplus\beta\otimes
x\boxtimes\left(  \mu\otimes y\right) \\
& =\alpha\otimes\lambda\otimes\left(  x\boxtimes y\right)  \oplus\beta
\otimes\mu\otimes\left(  x\boxtimes y\right)  ,
\end{align*}
from which the statement follows.
\end{proof}

\begin{corollary}
If $A$ and $B$ have finite eigenvectors, in particular if they are
irreducible, then $\lambda\left(  A\boxtimes B\right)  =\lambda\left(
A)\otimes\lambda(B\right)  .$
\end{corollary}

\begin{proof}
By Theorem \ref{Th Max eproblem} $\lambda\left(  A\right)  $ and
$\lambda\left(  B\right)  $ is the eigenvalue of $A$ and $B$ respectively with
associated finite eigenvectors, say $x$ and $y.$ By Theorem
\ref{Th Eproblem of TTP} $x\boxtimes y$ is a finite eigenvector of $A\boxtimes
B$ with the associated eigenvalue $\lambda\left(  A)\otimes\lambda(B\right)
.$ But the only eigenvalue of $A\boxtimes B$ associated with a finite
eigenvector is $\lambda\left(  A\boxtimes B\right)  $ (Theorem
\ref{Th Max eproblem}).
\end{proof}

If $A$ is a matrix then $vec\left(  A\right)  $ will stand for the vector
whose components are formed by the first column of $A$ followed by the second
column and so on.

\begin{theorem}
Matrix equation%
\begin{equation}
A_{1}\otimes X\otimes B_{1}\oplus A_{2}\otimes X\otimes B_{2}\oplus...\oplus
A_{r}\otimes X\otimes B_{r}=C,\label{Eq Matrix eq}%
\end{equation}
\bigskip where $A_{i},B_{i}$ and $C$ are of compatible sizes, is equivalent to
the vector-matrix system%
\[
\left(  A_{1}\boxtimes B_{1}^{T}\oplus A_{2}\boxtimes B_{2}^{T}\oplus...\oplus
A_{r}\boxtimes B_{r}^{T}\right)  \otimes vec\left(  X\right)  =vec\left(
C\right)  .
\]

\end{theorem}

\begin{proof}
It is sufficient to prove the statement for $r=1.$ Let $X_{1},X_{2},...$ be
the columns of $X.$ Let%
\[
B=\left(
\begin{array}
[c]{ccc}%
b_{11} & ... & b_{1p}\\
... & ... & ...\\
b_{n1} & ... & b_{np}%
\end{array}
\right)  .
\]
Then
\begin{align*}
vec\left(  A\otimes X\otimes B\right)   & =\\
& =vec\left(  \left(  A\otimes X_{1},...,A\otimes X_{n}\right)  \otimes
B\right) \\
& =\left(
\begin{array}
[c]{c}%
A\otimes X_{1}\otimes b_{11}\oplus A\otimes X_{2}\otimes b_{21}\oplus...\oplus
A\otimes X_{n}\otimes b_{n1}\\
A\otimes X_{1}\otimes b_{12}\oplus A\otimes X_{2}\otimes b_{22}\oplus...\oplus
A\otimes X_{n}\otimes b_{n2}\\
...\\
A\otimes X_{1}\otimes b_{1p}\oplus A\otimes X_{2}\otimes b_{2p}\oplus...\oplus
A\otimes X_{n}\otimes b_{np}%
\end{array}
\right) \\
& =\left(
\begin{array}
[c]{cccc}%
A\otimes b_{11} & A\otimes b_{21} & \cdots & A\otimes b_{n1}\\
A\otimes b_{12} & A\otimes b_{22} & \cdots & A\otimes b_{n2}\\
\cdots & \cdots & \cdots & \cdots\\
A\otimes b_{1p} & A\otimes b_{2p} & \cdots & A\otimes b_{np}%
\end{array}
\right)  \otimes\left(
\begin{array}
[c]{c}%
X_{1}\\
X_{2}\\
...\\
X_{n}%
\end{array}
\right) \\
& =\left(  A\boxtimes B^{T}\right)  \otimes vec\left(  X\right)
\end{align*}

\end{proof}

\begin{corollary}
\bigskip Matrix equation (\ref{Eq Matrix eq}) has a solution if and only if
\[
D\otimes\left(  D^{\#}\otimes^{\prime}vec\left(  C\right)  \right)
=vec\left(  C\right)  ,
\]
where
\[
D=A_{1}\boxtimes B_{1}^{T}\oplus A_{2}\boxtimes B_{2}^{T}\oplus...\oplus
A_{r}\boxtimes B_{r}^{T}.
\]

\end{corollary}

\begin{lemma}
\label{L diagonal matrices}If $P$ and $Q$ are diagonal matrices of order $n$
and $A\in\overline{\mathbb{R}}^{n\times n}$ then
\[
maper(A\boxtimes Q)=\left(  maper(A)\right)  ^{n}\otimes\left(
maper(Q)\right)  \bigskip^{n}
\]
$\ \ \ $ and
\[
maper(P\boxtimes A)=\left(  maper(P)\right)  ^{n}\otimes\left(
maper(A)\right)  ^{n}.
\]

\end{lemma}

\begin{proof}
$A\boxtimes Q$ is the blockdiagonal matrix%
\[
\left(
\begin{array}
[c]{ccc}%
A\otimes q_{11} & ... & ...\\
... & A\otimes q_{22} & ...\\
... & ... & \ddots
\end{array}
\right)  .
\]
Hence
\begin{align*}
maper(A\boxtimes Q)  & =maper(A\otimes q_{11})\otimes maper(A\otimes
q_{22})\otimes...\\
& =maper(A)\otimes q_{11}^{n}\otimes maper(A)\otimes q_{22}^{n}\otimes...\\
& =\left(  maper(A)\right)  ^{n}\otimes\left(  q_{11}\otimes q_{22}%
\otimes...\right)  \bigskip^{n}\\
& =\left(  maper(A)\right)  ^{n}\otimes\left(  maper(Q)\right)  \bigskip^{n}.
\end{align*}
The other statement is proved similarly, since $P\boxtimes A$ is also a
block-diagonal matrix.
\end{proof}

\begin{theorem}
If $A\in\overline{\mathbb{R}}^{n\times n},B\in\overline{\mathbb{R}}^{m\times
m}$ then
\[
maper\left(  A\boxtimes B\right)  =\left(  maper(A)\right)  ^{n}\otimes\left(
maper(B)\right)  ^{m}.
\]

\end{theorem}

\begin{proof}
Let $P,Q,R,S$ be diagonal matrices (see Theorem \ref{Th AP}) such that%
\begin{align*}
C  & \leq0\\
maper(C)  & =0
\end{align*}
where $C=P\otimes A\otimes Q$ and%
\begin{align*}
D  & \leq0\\
maper(D)  & =0
\end{align*}
where $D=R\otimes B\otimes S.$ Hence%
\[
maper(A)=\left(  maper(P)\otimes maper(Q)\right)  ^{-1}
\]
and%
\[
maper(B)=\left(  maper(R)\otimes maper(S)\right)  ^{-1}
\]

By two applications of Theorem \ref{ABCD is ACBD} we have%
\begin{align*}
C\boxtimes D  & =\\
& =\left(  P\otimes A\otimes Q\right)  \boxtimes\left(  R\otimes B\otimes
S\right)  =\\
& =\left(  P\boxtimes R\right)  \otimes\left(  \left(  A\otimes Q\right)
\boxtimes\left(  B\otimes S\right)  \right) \\
& =\left(  P\boxtimes R\right)  \otimes\left(  A\boxtimes B\right)
\otimes\left(  Q\boxtimes S\right)  .
\end{align*}
It is sufficient now to show that $maper(C\boxtimes D)=0$ because by Lemma
\ref{L diagonal matrices} and Theorem \ref{Th AP} then%
\begin{align*}
maper\left(  A\boxtimes B\right)   & =\left(  maper\left(  P\boxtimes
R\right)  \otimes maper\left(  Q\boxtimes S\right)  \right)  ^{-1}\\
& =\left(  \left(  maper(P)\right)  ^{n}\otimes\left(  maper(R)\right)
^{n}\otimes\left(  maper(Q)\right)  ^{n}\otimes\left(  maper(S)\right)
^{n}\right)  ^{-1}\\
& =\left(  maper(A)\right)  ^{n}\otimes\left(  maper(B)\right)  ^{m}.
\end{align*}
Clearly, $E=C\boxtimes D\leq0$ and so we only need to identify a permutation,
say $\tau\in P_{n^{2}}$ such that $e_{i,\tau\left(  i\right)  }=0$ for all
$i=1,...,n^{2}.$ Recall that%
\[
C\boxtimes D=\left(
\begin{array}
[c]{ccc}%
C\otimes d_{11} & C\otimes d_{12} & ...\\
C\otimes d_{21} & ... & \\
... &  &
\end{array}
\right)  .
\]
Since $C\leq0$ and $maper(C)=0$ there is a $\pi\in P_{n}$ such that
$c_{i,\pi(i)}=0$ for all $i\in N$ and similarly there is a $\sigma\in P_{n}$
such that $d_{i,\sigma(i)}=0$ for all $i\in N.$ Let $i\in\left\{
1,...,n^{2}\right\}  ,$ $i=kn+j,0\leq k\leq n-1,1\leq j\leq n.$ Set
$\tau\left(  i\right)  =\left(  \sigma\left(  k+1\right)  -1\right)
n+\pi\left(  j\right)  .$ Then%
\[
e_{i,\tau\left(  i\right)  }=c_{j,\pi\left(  j\right)  }\otimes d_{k+1,\sigma
\left(  k+1\right)  }=0\otimes0=0.
\]

\end{proof}

\end{document}